\documentclass[12pt]{amsart}
\usepackage{epsfig,color}

\usepackage{hyperref}
\hypersetup{
    colorlinks=true, 
    linktoc=all,     
    linkcolor=blue,  
}

\usepackage{mathrsfs}
\usepackage{amssymb}

\theoremstyle{definition}

\newtheorem{theorem}{Theorem}[section]
\newtheorem{definition}[theorem]{Definition}

\newtheorem{proposition}[theorem]{Proposition}
\newtheorem{lemma}[theorem]{Lemma}
\newtheorem{remark}[theorem]{Remark}
\newtheorem{corollary}[theorem]{Corollary}

\numberwithin{equation}{section}

\usepackage{enumerate}

\DeclareMathOperator*{\spa}{span}
\DeclareMathOperator*{\Gr}{Gr}
\DeclareMathOperator*{\supp}{spt}
\renewcommand\div{\operatorname{div}}
\DeclareMathOperator*{\im}{im}

\DeclareMathOperator*{\id}{id}

\newcommand{\wt}{\widetilde}
\newcommand{\pr}{\partial}

\newcommand{\Lap}{\Delta}

\newcommand{\Ric}{\operatorname{Ric}}

\DeclareMathOperator*{\Hess}{Hess}

\renewcommand*\d{\mathop{}\!\mathrm{d}}

\DeclareMathOperator*{\area}{area}

\newcommand{\sn}{\operatorname{sn}}

\usepackage{comment}

\def\MR#1{}

\title{Widths of balls and free boundary minimal submanifolds}

\author{Jonathan J. Zhu}
\address{Department of Mathematics, Princeton University, Fine Hall, Washington Road, Princeton, NJ 08544, USA}
\email{jjzhu@math.princeton.edu}

\begin{document}
\begin{abstract}
We observe that the $k$-dimensional width of an $n$-ball in a space form is given by the area of an equatorial $k$-ball. We also discuss the relationship between widths and lower bounds for the area of a free boundary minimal submanifold in a space form ball. 
\end{abstract}
\date{\today}
\maketitle

\section{Introduction}

Let $(M^n,g)$ be a Riemannian manifold - for simplicity, diffeomorphic to a ball. The width $\omega^k(M^n)$ is a min-max quantity involving families of $k$-cycles in $M$. Specifically, we may consider continuous maps $\Phi$ from a simplicial complex to the space $\mathcal{Z}_{k,rel}(M,\pr M; \mathbb{Z}_2)$ of (relative) $k$-cycles with mod 2 coefficients, and the maximal-area cycle amongst such a family. The ($k$-dimensional) width  is then given by the least maximal area across all families which \textit{sweep out} $M$ (cf. \cite{Gu07} or \cite{LMN}): 

\begin{equation}
\omega^k(M) = \inf_{\Phi^*\alpha \neq 0} \sup_{t} \area(\Phi(t)).
\end{equation}
The sweep-out condition $\Phi^*\alpha\neq 0$ may be understood as the condition that $\bigcup_t \Phi(t)$ covers $M$; technically, $\alpha$ is a generator of $H^{n-k}(\mathcal{Z}_{k,rel}(M ,\pr M; \mathbb{Z}_2); \mathbb{Z}_2) \simeq H^n(M, \pr M; \mathbb{Z}_2) \simeq \mathbb{Z}_2$. 

Let $B^n_{R; K}$ be the ball of radius $R$ in the space form $\mathbb{M}^n_K$ of constant curvature $K$. In this note, we observe that the $k$-dimensional width of $B^n_{R; K}$ is realised by the equatorial disk $B^{k}_{R; K}$:

\begin{theorem}
\label{thm:main}
Let $1\leq k <n$, $K\in \mathbb{R}$ and $R>0$. If $K>0$, further assume that $R\leq \frac{\pi}{2\sqrt{K}}$. Then $\omega^k(B^n_{R; K}) = \area(B^k_{R; K}).$
\end{theorem}

The main content is the lower bound, as the upper bound can be seen by explicit sweepout: Let $p$ be the centre of $B^n_{R;K} \subset \mathbb{M}^n_K$ and fix a totally geodesic $\mathbb{M}^{n-k}_K$ which passes through $p$. For each $y\in \mathbb{M}^{n-k}_K$, let $\Sigma_y$ be the totally geodesic $\mathbb{M}^k_K$ which intersects $\mathbb{M}^{n-k}_K$ orthogonally at $y$; then the $\Sigma_y\cap B^n_{R;K}$ give a sweepout of $B^n_{R;K}$ whose maximal slice is an equatorial $B^k_{R;K}$. 

The corresponding result for the interrelated notion of \textit{waists} was recently proven by Akopyan and Karasev \cite{AK20}. Gromov \cite{Gr83} originally introduced the idea of using families of cycles to quantify the size of a manifold. Almgren \cite{Al65} had previously developed min-max methods to construct minimal cycles; he showed, in considerable generality, that the width is realised by a stationary integral varifold - a weak version of a minimal submanifold. More recently, Guth proved fundamental estimates for $k$-dimensional widths \cite{Gu07} and the more general \textit{$p$-widths of dimension $k$} \cite{Gu09}. A striking Weyl law for the latter was proven by Liokumovich-Marques-Neves \cite{LMN}. 

Gromov had outlined two different approaches towards width/waist inequalities. In \cite{Gr83}, he observed that Almgren's work yields a simple proof that the width of the unit sphere is $\omega^k(\mathbb{S}^n) = \area(\mathbb{S}^k)$, as any (nontrivial) stationary $k$-varifold in $\mathbb{S}^n$ has area at least that of the unit $k$-sphere\footnote{For instance, the cone over a stationary varifold in $\mathbb{S}^n$ is stationary in $\mathbb{R}^{n+1}$, and the density of such a cone is at least 1 by the classical monotonicity formula.}. It is also claimed there that the $K=0$ case of Theorem \ref{thm:main} follows from the same `direct' method: one realises the width by a stationary varifold, then proves a sharp lower bound for such a varifold. On the other hand, in \cite{Gr10}, Gromov argues that the nonpositive curvature cases of Theorem \ref{thm:main} may be deduced from the width inequality for the sphere, by utilising a radial comparison map which contracts $k$-areas, but which preserves equatorial areas. 

In this note, we use comparison with the hemisphere (Section \ref{sec:comp}) to deduce Theorem \ref{thm:main}, as in the `comparison' approach outlined by Gromov and previously executed by Akopyan-Karasev \cite{AK20}. (The hemisphere inherits a width inequality from the sphere, by reflection.) We also fulfil the `direct' approach in certain cases (Section \ref{sec:direct}): Noting that for manifolds with boundary $(M,\pr M)$, the width is realised by a varifold which is stationary with respect to variations preserving the boundary, one requires an estimate for the area of these `free boundary stationary' varifolds in $(M,\pr M)$. The corresponding regular solutions are minimal submanifolds which meet the boundary orthogonally, and for such submanifolds of the Euclidean ball $B^n \subset \mathbb{R}^n$ an estimate was proven by Brendle \cite{Br12} using an ingenious variational argument:

\begin{theorem}[\cite{Br12}]
\label{thm:brendle}
Let $\Sigma$ be a $k$-dimensional minimal submanifold in $B^n$ whose boundary $\pr \Sigma \subset \pr B^n$ meets $\pr B^n$ orthogonally. Then $\area(\Sigma^k) \geq \area (B^k)$. Moreover, if equality holds, then $\Sigma^k$ is contained in a $k$-dimensional affine subspace. 
\end{theorem}

One technical component of this note is to extend Theorem \ref{thm:brendle} to the setting of integral, free boundary stationary varifolds in $(B^n, \pr B^n)$ - see Theorem \ref{thm:fb-estimate}. Note that if $2\leq k=n-1 \leq 6$, one can simply use the smooth version Theorem \ref{thm:brendle}, by invoking the free boundary min-max regularity theory of \cite{LZ}. We remark that Brendle's approach was extended to certain positive curvature cases by Freidin-McGrath \cite{FM20, FM19}, which should in turn generalise to the varifold setting. 

On the other hand, we observe that in codimension 1, one may deduce sharp lower bounds for free boundary minimal hypersurfaces by the width (Section \ref{sec:estimates}):

\begin{proposition}
\label{prop:estimate-intro}
Let $K\in \mathbb{R}$, $R>0$. If $K>0$, further assume that $R\leq \frac{\pi}{2\sqrt{K}}$. Then any nontrivial embedded free boundary minimal hypersurface $\Sigma^{n-1}$ in $(B^n_{R; K}, \pr B^n_{R; K})$ satisfies $\area(\Sigma) \geq \area( B^{n-1}_{R; K})$. 
\end{proposition} 

This bound also implies a sharp inequality for the isoperimetric ratio when $K\leq 0$ (see Corollary \ref{cor:iso}). To the author's knowledge, these bounds are new for $K<0$, and indeed it does not seem clear how to generalise other direct approaches such as \cite{Br12, FM20, FM19} and \cite{FS11, FS13} to the negative curvature setting. We remark that the `direct' method seems to favour positive curvature, whilst the `comparison' method appears to favour negative curvature. Given Proposition \ref{prop:estimate-intro}, however, we suggest that it is worth re-examining whether the analogous lower bounds hold in any codimension (as in Theorem \ref{thm:brendle}).

We also mention the related notion of `waists'. The terminology is not entirely standard, but here we distinguish waists as min-max quantities for the fibres of sufficiently regular maps $M^n \to \mathbb{R}^{n-k}$, which measure the $n$-volume of neighbourhoods of the fibres, or in the limit, their $k$-dimensional Minkowski content. Gromov \cite{Gr03} initially proved a waist inequality for $M=\mathbb{R}^n$ with the Gaussian measure, and sketched a proof (later detailed by Memarian \cite{Me11}) of a sharp waist inequality for the round sphere $\mathbb{S}^n$. For the $k$-dimensional content, Akopyan-Karasev \cite{AK} used orthogonal projection to deduce a sharp waist inequality for $B^n$ from the Gromov-Memarian estimate and later used the comparison approach to establish the sharp waist inequality in space form balls \cite{AK20}. We remark that their comparison map is described via a conformal presentation, but should be equivalent to the warped product presentation used here. 

Finally, we note that Gromov \cite{Gr83} had also previously proven an elementary lower bound $\omega^k(B^n) \geq c(n)>0$ (see also Guth \cite{Gu09}).


\subsection*{Acknowledgements}

This work was supported in part by the National Science Foundation under grant DMS-1802984. The author would like to thank Antoine Song, as well as Kevin Sackel and Umut Varolgunes, for their encouragement and stimulating discussions. The author also thanks Larry Guth, Peter McGrath, Keaton Naff, Ao Sun and the anonymous referees for pointing them to helpful references and other valuable suggestions. 

\section{Preliminaries}

As above we denote by $B^k_{R; K}(p)$ the geodesic ball in the space form $\mathbb{M}^k_K$ of constant curvature $K$, of radius $R$ centred at $p$. We will frequently suppress aspects of the notation when clear from context: When the radius is omitted it can be taken as $R=1$; when the curvature $K$ is omitted it can be taken as $K=0$; when the point $p$ is omitted it can be taken as the origin; if the dimension is suppressed, it should be taken to be the ambient dimension $n$. For instance, $A_{s,t}(x) = B_t(x)\setminus \overline{B_s(x)} = B^n_t(x)\setminus \overline{B^n_s(x)}$ denotes the Euclidean annulus about $x\in \mathbb{R}^n$. We denote the area of the unit Euclidean $k$-ball by $\alpha_k = \area(B^k)$. 

We denote by $\mathbb{S}^k(r)$ the round $k$-sphere of radius $r$, and $\mathbb{S}^k_+(r)$ the corresponding hemisphere. We denote the area of the unit round $k$-sphere by $\beta_k = \area(\mathbb{S}^k)$. 

Let $\mathfrak{X}(\mathbb{R}^N)$ be the space of vector fields on $\mathbb{R}^N$. For a submanifold $M^n \subset \mathbb{R}^N$ we define \[\mathfrak{X}(M)= \{X\in \mathfrak{X}(\mathbb{R}^N) | \forall p \in M, X(p) \in T_p M \},\] \[\mathfrak{X}(M, \pr M)= \{X\in \mathfrak{X}(M) | \forall p \in \pr M, X(p) \in T_p \pr M\}.\] The latter is the space of tangent vector fields on $M$ that are also tangent along its boundary. 

\subsection{Submanifolds}

Let $M^n$ be a compact Riemannian manifold with boundary $\pr M$. A submanifold $\Sigma$ is said to have \textit{free boundary} if $\pr \Sigma \subset \pr M$ meets $\pr M$ orthogonally. The submanifold $\Sigma$ is minimal if its mean curvature (vector) vanishes. 

\subsection{Varifolds}
We recall some basic definitions for varifolds; the reader is directed to \cite{Si83} for a comprehensive account or \cite{LZ} for a summary. 

The Grassmannian of (unoriented) $k$-planes in $\mathbb{R}^n$ is denoted $\Gr(k,n)$. The Grassmannian bundle over $A\subset \mathbb{R}^n$ is denoted $\Gr_k(A) = A\times \Gr(k,n)$. The space $\mathcal{V}_k(\mathbb{R}^n)$ of $k$-varifolds in $\mathbb{R}^n$ is the set of all Radon measures on $\Gr_k(\mathbb{R}^n)$, equipped with the weak topology. The support $\supp \|V\|$ is the smallest closed subset $\mathcal{K} \subset \mathbb{R}^n$ such that the restriction of $V$ to $\Gr_k(\mathbb{R}^n\setminus \mathcal{K})$ is zero. We denote the space of integer rectifiable (integral) varifolds by $\mathcal{IV}_k(\mathbb{R}^n)$ (see \cite{Si83}). The weight of $V$ is the measure $\|V\|=\mu_V$ and the mass is $\mathbf{M}(V)= \|V\|(\mathbb{R}^n)$. Any smooth map $\phi:\mathbb{R}^n\to \mathbb{R}^n$ gives rise to a pushforward map $\phi_\#: \mathcal{V}_k(\mathbb{R}^n) \to \mathcal{V}_k(\mathbb{R}^n)$. 

Let $M^n$ be a compact Riemannian manifold with boundary, which we may isometrically embed in $\mathbb{R}^N$. The space $\mathcal{IV}_k(M)$ of integral $k$-varifolds on $M$ consists of those $V\in\mathcal{IV}_k(\mathbb{R}^N)$ with $\supp \|V\|\subset \overline{M}$. 

\subsubsection{Stationary varifolds}

Let $\mathfrak{X}_0(\mathbb{R}^n)$ be the space of compactly supported vector fields on $\mathbb{R}^n$. Any smooth $X\in \mathfrak{X}_0(\mathbb{R}^n)$ generates a one-parameter family of diffeomorphisms $\phi_t:\mathbb{R}^n\to \mathbb{R}^n$, and a varifold $V\in \mathcal{V}_k(\mathbb{R}^n)$ is \textit{stationary} if 
\begin{equation}\label{eq:stat0}\delta V(X) = \left.\frac{\d}{\d t}\right|_{t=0} \mathbf{M}((\phi_t)_\# V) = \int\div_S X(x) \d V(x,S) =0\end{equation} for all such $X$. 

Similarly, we say that a varifold $V\in \mathcal{IV}_k(M)$ is \textit{free boundary stationary} if 
\begin{equation}
\label{eq:stat}
\int\div_S X(x) \d V(x,S) =0
\end{equation}
 holds for any smooth $X \in \mathfrak{X}(M,\pr M)$ - that is, tangent along $\pr M$. By standard approximation arguments, it is equivalent that (\ref{eq:stat}) holds for all \textit{Lipschitz} vector fields $X \in \mathfrak{X}(M,\pr M)$. We denote the space of (integral) free boundary stationary varifolds by $\mathcal{SV}_k(M,\pr M)$.

\subsection{Widths}

For precise definitions of min-max widths, we refer the reader to \cite{Gu09} or \cite{LMN}, in which more general widths are defined. Note that what we refer to as the $k$-dimensional width $\omega^k$ is, in the terminology of \cite{LMN}, the 1-parameter width $\omega^k_1$ of dimension $k$. For our technical purposes, it is sufficient to know that the widths are realised by stationary integral varifolds, which was proven by Almgren \cite{Al65}:

\begin{proposition}[\cite{Al65}]
\label{thm:min-max}
Let $M^n$ Riemannian manifold with smooth boundary $\pr M$. For any $1\leq k<n$, the $k$-dimensional width is realised by a nontrivial, integral, free boundary stationary varifold $V$ in $(M ,\pr M)$. That is, there exists $V\in\mathcal{SV}_k(M, \pr M)$ such that \[ \omega^k(M)=\mathbf{M}(V) .\] 
\end{proposition}

\section{Widths by comparison}
\label{sec:comp}

In this section, we give a proof of Theorem \ref{thm:main} for all $K$, following the approach outlined by Gromov in \cite{Gr10}. In particular, we use the width of the hemisphere together with radial $k$-area-contracting maps to deduce the estimate for the other balls. We remark that the same method should also yield waist inequalities by comparison. 

It will be convenient to consider warped products:

\begin{definition}
\label{def:warped}
Consider smooth functions $h:[0,R]\to [0,\infty)$ satisfying: \begin{enumerate}
\item[($\dagger$)] $h(0)=0$ and $h(r)/r$ is a smooth function of $r^2$. 
\end{enumerate}

Note that this condition ensures that $r\mapsto h(r)$ defines a smooth map $\mathbb{R}^n\to \mathbb{R}^n$. We define $M^n_{h,R}$ to be the warped product $[0,R]\times_h \mathbb{S}^{n-1}$ (again note that the metric $g=\d r^2 + h(r)^2 g_{\mathbb{S}^{n-1}}$ is smooth at $r=0$). 

We say that $M$ has \textit{(at-least) equatorial widths} if $\omega^k(M^n_{h,R}) \geq \area(M^k_{h,R})$. 
\end{definition}

Note that $\area(M^k_{h,R}) = \area(\mathbb{S}^{k-1}) \int_0^R h(r)^{k-1} \d r$. 

Note also that the geodesic ball $B^n_{R; K}$ is isometric to $M^n_{\sn_K, R}$, where
\begin{equation}
\sn_K(r) = 
\begin{cases}
\frac{1}{\sqrt{K}}\sin(r\sqrt{K}), & K>0 \\
r ,& K=0\\
\frac{1}{\sqrt{|K|}} \sinh(r \sqrt{|K|}), & K<0.
\end{cases}
\end{equation}

Theorem \ref{thm:main} is thus equivalent to the statement that $M^n_{\sn_K,R}$ has equatorial widths, where $R< \frac{\pi}{2\sqrt{K}}$ if $K>0$. (Recall that the upper bound is clear from the totally geodesic foliation.) 

\subsection{Width of the hemisphere} 
\label{sec:hemisphere}

First, we obtain the width of the hemisphere by reflection:

\begin{proof}[Proof of Theorem \ref{thm:main} ($K>0, R= \frac{\pi}{2\sqrt{K}}$)]
By scaling we may assume $K=1$. Any sweepout of the hemisphere $\mathbb{S}^n_+$ induces a sweepout of the sphere $\mathbb{S}^n$ by reflection. Since the sphere is known to have equatorial widths, we have $\omega^k(\mathbb{S}^n_+) \geq \frac{1}{2} \omega^k(\mathbb{S}^n) = \frac{1}{2} \area(\mathbb{S}^k) = \area(\mathbb{S}^k_+)$.
\end{proof}

\subsection{Comparison for warped products}

For certain warped products, we can pull back equatorial widths by constructing a $k$-area contracting map. 

\begin{proposition}
\label{prop:comp}
Let $1\leq k<n$ and consider smooth functions $h_0: [0,R_0] \to [0,\infty)$, $h_1: [0,R_1] \to [0,\infty)$, $f:[0,R_0]\to[0,R_1]$ satisfying ($\dagger$). Further suppose that $f' \geq 1$, $R_1=f(R_0)$ and $h_0(r)^{k-1} = f'(r) h_1(f(r))^{k-1}$. 

If the warped product ball $M^n_{h_1,R_1}$ has equatorial widths, then so does $M^n_{h_0, R_0}$. 
\end{proposition}
\begin{proof}
We may define a smooth map $\phi: M^n_{h_0, R_0} \to M^n_{h_1, R_1}$ by sending $r\mapsto f(r)$. The metric on $M^n_{h_i,R_i}$ is $g_i = \d r^2 + h_i(r)^2 g_{\mathbb{S}^{n-1}}$, so the pullback metric is $\phi^*g_1 = f'(r)^2 \d r^2 + h_1(f(r))^2 g_{\mathbb{S}^{n-1}}$. Therefore, to check that $\phi$ contracts $k$-areas, it suffices to note that 
\begin{enumerate}
\item $f'(r)h_1(f(r))^{k-1} \leq h_0(r)^{k-1}$; and 
\item $h_1(f(r)) \leq h_0(r)$. 
\end{enumerate}
The first item is actually an equality by supposition. The second then follows as $f'\geq 1$. 

Now since $\phi$ contracts $k$-areas, it is clear from the definition of width that $\omega^k(M^n_{h_0,R_0}) \geq \omega^k(M^n_{h_1,R_1})$. But since $M^n_{h_1,R_1}$ has equatorial widths, we have 
\begin{equation}
\begin{split}
 \omega^k(M^n_{h_1,R_1}) &\geq \area(M^k_{h_1,R_1}) =\area(\mathbb{S}^{k-1}) \int_0^{R_1} h_1(r)^{k-1} \d r \\& = \area(\mathbb{S}^{k-1}) \int_0^{R_0} h_0(r)^{k-1} \d r = \area(M^k_{h_0,R_0}),\end{split}
 \end{equation}
 where to get to the second line we have used the integral substitution $r \leftrightarrow f(r)$. This completes the proof. 
\end{proof}

\subsection{Widths of space form balls} We may now deduce the main theorem by comparison with the hemisphere: 

\begin{proof}[Proof of Theorem \ref{thm:main}]
We have already shown that the hemisphere has equatorial widths (Section \ref{sec:hemisphere}), so by scaling $\mathbb{S}^n_+(\sqrt{K_1})= M^n_{h_1, R_1}$ does too for any $K_1>0$. Here $h_1=\sn_{K_1}$ and $R_1=\frac{\pi}{2\sqrt{K_1}}$. The aim is to apply Proposition \ref{prop:comp} to deduce the result for the other space form balls. To do so, 
we will consider $K < K_1$, $h_0=\sn_K$ and define $f$ by \begin{equation}\label{eq:def-f}\int_0^s \sn_K^{k-1}(r)\d r= \int_0^{f(s)} \sn^{k-1}_{K_1}(r) \d r ,\end{equation} so that $f$ is the unique solution of $f'(r) \sn_{K_1}(f(r))^{k-1} = \sn_K^{k-1}(r)$, $f(0)=0$. Note that from this characterisation, it is clear that $f$ satisfies $(\dagger$). 

We also note that $R_1 \in \im f$: If $K\leq 0$, then the left hand side is unbounded, whilst if $K>0$ then \[\int_0^{\frac{\pi}{2\sqrt{K}}}  \sn_K^{k-1}(r)\d r = \frac{1}{\sqrt{K}^k} \int_0^\frac{\pi}{2} \sin^k(r)\d r > \int_0^{\frac{\pi}{2\sqrt{K_1}}}  \sn_{K_1}^{k-1}(r)\d r .\]

In fact, this implies that $R_0 = f^{-1}(R_1) < \frac{\pi}{2\sqrt{K}}$.

 In order to apply Proposition \ref{prop:comp}, it remains only to check that $f'(r)\geq 1$ for $r\in [0, R_0]$. The proof of this fact is deferred to Lemma \ref{lem:df-1}. Given that Proposition \ref{prop:comp} applies, we choose parameters depending on the sign of $K$: 

Case 1: $K=0$. By scaling, it suffices to show that $B^{n}_{R_0; 0}$ has equatorial widths for any particular $R_0>0$. So take $K_1=1$; then $R_0 =f^{-1}(\frac{\pi}{2})$, $h_0=r$. Proposition \ref{prop:comp} now implies that $M^n_{h_0, R_0} = B^{n}_{R_0; 0}$ has equatorial widths. 

Case 2: $K>0$. By scaling, it suffices to show that for each $\alpha \in (0,\frac{\pi}{2})$, there is some $K>0$ such that $B^{n}_{\frac{\alpha}{\sqrt{K}}; K}$ has equatorial widths. Take $K_1=1$ and $R_0 = \frac{\alpha}{\sqrt{K}}$ where \[K = \left( \frac{\int_0^\alpha \sin^{k-1}(r) \d r}{\int_0^\frac{\pi}{2} \sin^{k-1}(r) \d r} \right)^{\frac{2}{k}}\in(0,1)\] so that $f(\frac{\alpha}{\sqrt{K}}) = \frac{\pi}{2}$ by (\ref{eq:def-f}). With $h_0 = \sn_K$, Proposition \ref{prop:comp} implies that $M^n_{h_0, R_0} = B^{n}_{\frac{\alpha}{\sqrt{K}}; K}$ has equatorial widths. 

Case 3: $K<0$. By scaling, it suffices to take $K=-1$ and show that $B^{n}_{R; -1}$ has equatorial widths for all $R>0$. We set \[ K_1 = \left(\frac{\int_0^{\frac{\pi}{2}} \sin^{k-1}(r) \d r}{\int_0^R \sinh^{k-1}(r) \d r} \right)^{\frac{2}{k}}>0,\] so that $f(R) = \frac{\pi}{2\sqrt{K_1}}$ by (\ref{eq:def-f}). Thus with $h_0 = \sinh$ and $R_0=R$, Proposition \ref{prop:comp} now implies that $M^n_{h_0, R_0} = B^{n}_{R; -1}$ has equatorial widths. 

\end{proof}

\begin{lemma}
\label{lem:df-1}
Let $K_1>0$, $h_1= \sn_{K_1}$, $R_1 = \frac{\pi}{2\sqrt{K_1}}$ and take $K < K_1$, $h_0 = \sn_K$. Define $f$ by \[\int_0^s \sn_K^{k-1}(r)\d r= \int_0^{f(s)} \sn_{K_1}^{k-1}(r) \d r\]
and set $R_0 = f^{-1}(R_1)$. Then $f' \geq 1$ on $[0,R_0]$. 
\end{lemma}
\begin{proof}
If $k=1$, then $f(r)=r$ and the statement is trivial, so henceforth we assume $k\geq 2$. Define $F(r) = \frac{h_1(f(r))}{h_0(r)}$ so that $f'(r) = F(r)^{1-k}$; it suffices to show that $F\leq1$. By the definition of $f$, and since $\sn \sim r$ for small $r$, we have $f(r) \sim r$ and hence $F(r) \to 1$ as $r\to 0$. 

Suppose now, for the sake of contradiction, that $M:= \sup_{[0,R_0]} F >1$. Then $M = F(r_*)$ for some $r_* >0$. A straightforward calculation gives \[F'(r) = \frac{h_1'(f(r))f'(r)}{h_0(r)} - \frac{h_1(f(r)) h_0'(r)}{h_0(r)^2}.\]

The first derivative test gives $F'(r_*)\geq 0$ (with equality if $r_*<R_0$), hence $\frac{h_1'(f(r_*))}{h_0'(r_*)} \geq M^k$.

If $K\leq 0$, then this gives $M^k\leq\frac{h_1'(f(r_*))}{h_0'(r_*)} = \frac{\cos(f(r_*)\sqrt{K_1})}{\cosh (r_*\sqrt{|K|})} \leq 1 < M$, which is already absurd. 

On the other hand, suppose $K>0$. As above, we have $R_0 < \frac{\pi}{2\sqrt{K}}$, so in particular $F'(R_0)= \frac{\cos(R_1\sqrt{K_1})}{\cos (R_0\sqrt{K})}=\frac{\cos(\frac{\pi}{2})}{\cos (R_0\sqrt{K})}=0$. Then the first and second derivative tests apply (even if $r_*= R_0$) to give $F'(r_*)=0, F''(r_*) \leq 0$. Using $F'(r_*)=0$ gives $\frac{h_1'(f(r_*))}{h_0'(r_*)} =M^k$, and together with $F(r_*)=M$ these imply that \[F''(r_*) = \frac{M^{2-2k} h_1''(f(r_*)) - Mh_0''(r_*)}{h_0(r_*)}.\] But $h_0'' = -Kh_0$, $h_1'' =-K_1 h_1$, so the second derivative test gives \[0 \geq F''(r_*) = -K_1 M^{3-2k} + KM, \] hence $M^{2(k-1)} \leq  \frac{K_1}{K}$. But note that $f'\geq \frac{1}{M^{k-1}}$, so after integrating we have \[ \frac{f(r)}{r} \geq \frac{1}{M^{k-1}} \geq \left(\frac{K}{K_1}\right)^\frac{1}{2}.\] On the other hand, the first derivative test gave $\frac{h_1'(f(r_*))}{h_0'(r_*)} = \frac{\cos(f(r_*)\sqrt{K_1})}{\cos (r_*\sqrt{K})}=M^k >1$. Since $\cos$ is decreasing on $[0,\frac{\pi}{2}]$, we must have $f(r_*) \sqrt{K_1} < r_*\sqrt{K}$. Comparing to the lower bound for $\frac{f(r)}{r}$, we must have $\left(\frac{K}{K_1}\right)^\frac{1}{2}  < \left(\frac{K}{K_1}\right)^\frac{1}{2}$, which is absurd. 
\end{proof}

\section{Widths by the direct method}
\label{sec:direct}

In this section, we give a proof of Theorem \ref{thm:main} in the case $K=0$ which fulfils the direct approach outlined by Gromov in \cite{Gr83}. That is, we combine the min-max realisation of the width, Theorem \ref{thm:min-max}, with a lower bound for the area of any free boundary stationary varifold in $(B^n, \pr B^n)$. This lower bound, Theorem \ref{thm:fb-estimate}, generalises the results of Brendle \cite{Br12} to the varifold setting. We remark that this approach should in fact work for certain positive curvature cases, by generalising the estimates of Freidin-McGrath \cite{FM20, FM19}, for free boundary minimal surfaces in spherical balls, to stationary varifolds. We have chosen to detail the proof only for the Euclidean ball for simplicity and clarity of exposition. 

\subsection{Monotonicity and consequences}
\label{sec:mono}

A boundary monotonicity formula for free boundary stationary varifolds was proven in \cite{GLZ}; we state a special case for Euclidean domains:

\begin{proposition}[\cite{GLZ}]
\label{prop:monotonicity}
Let $\Omega^n$ be a compact domain in $\mathbb{R}^n$ with smooth boundary $\pr \Omega$. There exists $r_0= r_0(\Omega) \in (0,\infty]$ such that if $V \in \mathcal{SV}_k(\Omega)$ is a free boundary stationary varifold in $(\Omega,\pr\Omega)$, then for any $y\in \pr\Omega$ and any $0< s<t< r_0$, we have 
\begin{equation}
\frac{\|V\|(B_{t}(y))}{t^k}-\frac{\|V\|(B_{s}(y))}{s^k} \geq \int_{\Gr_k(A_{s,t}(y))} \frac{|D_S^\perp \rho|^2}{(1+\gamma \rho)\rho^k} \d V(x,S),
\end{equation}
where $\rho(x) = |x-y|$ and $\gamma = \frac{1}{r_0}$. 
\end{proposition}

A similar boundary monotonicity was proven in \cite{GJ}. With such a monotonicity formula in hand, it follows that the \textit{density} \[\Theta^k(\|V\|,x) = \lim_{r\to 0} \frac{\|V\|(B_r(x))}{\alpha_k r^k}\] is well-defined everywhere, and since $V$ is integral, $\Theta^k(\|V\|,x)$ is an integer $\mu_V$-a.e. Moreover, the modified density \[\wt{\Theta}^k(\|V\|,x) := \begin{cases} \Theta^k(\|V\|,x), &x\notin \pr \Omega \\ 2\Theta^k(\|V\|,x), & x\in \pr \Omega\end{cases}\] is upper semicontinuous (see for instance \cite[Section 6.2]{dlR}). It follows that $\wt{\Theta}^k(\|V\|,x) \geq 1$ at every point $x\in\supp \|V\|$. 

\subsection{Stationary free boundary varifolds in the unit ball}

Here we extend Brendle's proof of Theorem \ref{thm:brendle} to the varifold setting. In particular, we will consider deformations by his cleverly constructed vector field 
\begin{equation} 
\label{eq:Y}
Y(x) = \frac{x}{2} - \frac{x-y}{|x-y|^k} - \frac{k-2}{2} \int_0^1 \frac{tx-y}{|tx-y|^k} \d t,
\end{equation}
where $y\in \pr B^n$. Its properties are summarised in the following lemma. 

Given a $k$-plane $S\subset \mathbb{R}^n$ we will write $\pi_S$ for the projection to $S$, and $\pi_S^\perp = \id - \pi_S$ the projection to the orthocomplement. For a function $f$ we set $D_S^T = \pi_S Df$ and $D_S^\perp f = \pi_S^\perp Df$. 

\begin{lemma}[\cite{Br12}]
\label{lem:vector-field}
Fix $1\leq k \leq n$. There is a non-decreasing function $h:[0,2]\to [0,\infty)$ such that for any $y\in \pr B^n$, the vector field $Y$ defined by (\ref{eq:Y}) is smooth on $\overline{B^n}\setminus \{y\}$ and satisfies:
\begin{enumerate}[(i)]
\item $\div_S Y(x) \leq \frac{k}{2} - \frac{k}{|x-y|^{k+2}}|\pi_{S}^\perp(x-y)|^2$ for any $(x,S) \in \Gr_k(\overline{B^n}\setminus\{y\})$; 
\item $Y|_{\pr B^n\setminus \{y\}}$ is tangent to $\pr B^n$;
\item $ \left| Y(x) + \frac{x-y}{|x-y|^k} \right| \leq \frac{h(|x-y|)}{|x-y|^{k-1}}$ for $x\in \overline{B^n}\setminus \{y\}$;
\item $\lim_{t\to 0} h(t)=0$. 
\end{enumerate}
\end{lemma}

\begin{remark}
Strictly, the proof of \cite[Lemma 8]{Br12} does not cover the case $k=1$. However, in this case note that $Y(x) + \frac{x-y}{|x-y|} = \frac{x}{2} + \frac{1}{2} \int_0^1 \frac{tx-y}{|tx-y|} \d t.$ By dominated convergence, $\int_0^1 \frac{tx-y}{|tx-y|} \d t \to -\frac{y}{|y|}=-y$ as $x\to y$, so still $Y(x) + \frac{x-y}{|x-y|}\to 0$ as desired. 
\end{remark}

We now proceed to the main estimate:

\begin{theorem}
\label{thm:fb-estimate}
Let $V\in \mathcal{SV}_k(B^n)$ be a nontrivial, integral, free boundary stationary $k$-varifold in the unit ball $(B^n, \pr B^n)$. Then \[ \mathbf{M}(V)  \geq \area(B^k).\] Moreover, if equality holds, then $\supp\|V\|$ is contained in a $k$-dimensional subspace. 
\end{theorem}
\begin{proof}


Let $\mathcal{K}=\supp\|V\|\cap \pr B^n$, so that $V$ is stationary in $\mathbb{R}^n\setminus \mathcal{K}$. As $V$ has compact support, by the convex hull property \cite[Theorem 19.2]{Si83}, $\supp\|V\|$ is contained in the convex hull of $\mathcal{K}$. In particular, $\mathcal{K}$ cannot be empty. Note that as in Section \ref{sec:mono}, any $y\in \mathcal{K}$ satisfies $\Theta^k(\|V\|,y)\geq \frac{1}{2}$.

So fix $y \in K$ and set $\rho(x) = |x-y|$. We note that $D\rho = \frac{x-y}{|x-y|}$. Now let $r,\epsilon>0$ such that $(1+\epsilon)r<r_0$, where $r_0$ is as in Proposition \ref{prop:monotonicity}. Consider the piecewise linear cutoff function \[\eta_{r,\epsilon} = \begin{cases} 0, & t\leq r \\ \frac{t-r}{\epsilon r} ,& r\leq t \leq (1+\epsilon)r \\ 1 ,& t\geq (1+\epsilon)r \end{cases}.\] 

Henceforth we write $\eta=\eta_{r,\epsilon}$, suppressing the dependence on $r,\epsilon$. 

Using Lemma \ref{lem:vector-field}(ii), the vector field $\eta(\rho)Y$ is tangent along $\pr B^n$. So since $V$ is free boundary stationary, we may apply (\ref{eq:stat}) to find
\begin{equation}
\label{eq:fb-1}
\begin{split}
0&= \int \div_S(\eta(\rho)Y) \d V(x,S)
\\&= \int \left( \eta(\rho) \div_S Y + \eta'(\rho)\langle D_S^T \rho, Y\rangle \right) \d V(x,S).
\end{split}
\end{equation}

Recall that Lemma \ref{lem:vector-field}(iii) gives $\left| Y + \frac{D\rho}{\rho^{k-1}}\right| \leq \frac{h(\rho)}{\rho^{k-1}}$. Then since $|D\rho|=1$, 
\begin{equation}
\begin{split}
\langle \pi_S D\rho, Y\rangle &\leq - \rho^{1-k}\langle D_S^T \rho, D\rho\rangle + \rho^{1-k}h(\rho)
\\& = -\rho^{1-k}(1- |D_S^\perp \rho|^2) + \rho^{1-k}h(\rho).
\end{split}
\end{equation}

Therefore we have
\begin{equation}
\label{eq:fb-1}
\begin{split}
\int \eta'(\rho) \rho^{1-k} \d V(x,S) \leq & \int \eta(\rho)\div_S Y \d V(x,S) + \int \eta'(\rho) \rho^{1-k} |D_S^\perp \rho|^2 \d V(x,S) \\&+ \int \eta'(\rho) \rho^{1-k} h(\rho) \d V(x,S). 
\end{split}
\end{equation}

We proceed to estimate each of the terms in (\ref{eq:fb-1}).

First by Lemma \ref{lem:vector-field}(i) we have \begin{equation}\int \eta(\rho)\div_S Y \d V(x,S) \leq \frac{k}{2} \mathbf{M}(V).\end{equation}

By boundary monotonicity, Proposition \ref{prop:monotonicity} we have \begin{equation*}
\begin{split}
\int \eta'(\rho) \rho^{1-k} \d V(x,S) &=\frac{1}{\epsilon r}\int_{A_{r,(1+\epsilon)r}(y)} \rho^{1-k} \d \|V\| \geq  \frac{1}{\epsilon }r^{-k}(1+\epsilon)^{1-k} \|V\|(A_{r,(1+\epsilon)r}(y))
\\& \geq \frac{1}{\epsilon }(1+\epsilon)^{1-k} \left((1+\epsilon)^k 
 -1\right) \frac{\|V\|(B_r(y))}{r^k}.
\end{split}
\end{equation*}

Taking the limit as $r\to 0$, we find 
\begin{equation}
\label{eq:fb-2}
\lim_{r\to 0} \int \eta'(\rho) \rho^{1-k} \d V(x,S) \geq (1+\epsilon)^{1-k} \frac{(1+\epsilon)^k-1}{\epsilon} \alpha_k\Theta^k(\|V\|,y). 
\end{equation}

Again by Proposition \ref{prop:monotonicity} we have \begin{equation*}
\begin{split}
\int \eta'(\rho) \rho^{1-k} |D_S^\perp \rho|^2 \d V(x,S) &\leq \frac{1}{\epsilon r} \int_{\Gr_k(A_{r,(1+\epsilon)r}(y))} \frac{|D_S^\perp\rho|^2}{\rho^{k-1}} \d V(x,S) 
\\&\leq \frac{1}{\epsilon }(1+\epsilon) (1+\gamma r(1+\epsilon))  \int_{\Gr_k(A_{r,(1+\epsilon)r}(y))} \frac{|D_S^\perp\rho|^2 
}{\rho^{k}(1+\gamma \rho)} \d V(x,S)  
\\&= \frac{1}{\epsilon }(1+\epsilon) (1+\gamma r(1+\epsilon)) \left( 
\frac{\|V\|(B_{(1+\epsilon)r}(y))}{r^k(1+\epsilon)^k} -
 \frac{\|V\|(B_r(y))}{r^k}\right). 
\end{split}
\end{equation*}

Taking the limit as $r\to 0$, we find 
\begin{equation}
\label{eq:fb-3}
\lim_{r\to 0}\int \eta'(\rho) \rho^{1-k} |D_S^\perp \rho|^2 \d V(x,S) \leq \frac{1+\epsilon}{\epsilon}(\alpha_k\Theta^k(\|V\|,y)-\alpha_k\Theta^k(\|V\|,y))=0. 
\end{equation}

Finally, we have \begin{equation*}\int \eta'(\rho) \rho^{1-k} h(\rho) \d V(x,S) =\frac{1}{\epsilon r} \int_{A_{r,(1+\epsilon)r}(y)}  \rho^{1-k} h(\rho) \d\|V\| \leq \frac{r^{-k}}{\epsilon} h(r) \|V\|(B_{(1+\epsilon)r}(y)). \end{equation*}

Taking the limit as $r\to 0$, by Lemma \ref{lem:vector-field}(iv) we have
\begin{equation}
\label{eq:fb-4}
\lim_{r\to 0}\int \eta'(\rho) \rho^{1-k} h(\rho) \d V(x,S)  \leq  \frac{(1+\epsilon)^k}{\epsilon}\alpha_k\Theta^k(\|V\|,y) \lim_{r\to 0} h(r) =0.
\end{equation}

Thus, taking the $r\to 0$ limit in (\ref{eq:fb-1}) and using the estimates (\ref{eq:fb-2}-\ref{eq:fb-4}), we have
\begin{equation}
0\leq \frac{k}{2}\mathbf{M}(V) - (1+\epsilon)^{1-k} \frac{(1+\epsilon)^k-1}{\epsilon} \alpha_k\Theta^k(\|V\|,y). 
\end{equation}

Taking the limit as $\epsilon\to 0$ then gives \begin{equation}
0\leq \frac{k}{2} \mathbf{M}(V) - k\alpha_k \Theta^k(\|V\|,y),\end{equation} which implies the desired estimate as $\Theta^k(\|V\|,y)\geq \frac{1}{2}$. 

Finally, if $\mathbf{M}(V) = \alpha_k$, then from the above and Lemma \ref{lem:vector-field}(i), we must have $\pi_S^\perp(x-y)=0$ for $V$-a.e. $(x,S)$, and any $y\in \mathcal{K}$. Then $\pi_S^\perp(y_1-y_2)=0$ for $V$-a.e. $(x,S)$ and any $y_1,y_2\in \mathcal{K}$. Therefore the affine span $S_\mathcal{K} = \spa \{y_1-y_2 | y_1,y_2\in \mathcal{K}\}$ satisfies $S_\mathcal{K} \subset S$ for $V$-a.e. $(x,S)$. 

But recall that $\supp \|V\|$ is contained in the convex hull of $\mathcal{K}$. Since $\mathbf{M}(V) \neq 0$, it follows that the convex hull of $\mathcal{K}$, and hence $S_\mathcal{K}$, must be at least $k$-dimensional. In particular it must be exactly $k$-dimensional, and $S = S_\mathcal{K}$ for $V$-a.e. $(x,S)$. This implies that $\supp \|V\|$ is contained in a $k$-dimensional affine subspace (parallel to $S_K$). 
\end{proof}

\begin{remark}
In Brendle's approach, the vector field $Y$ corresponds to the gradient of the Green's function for the ($k$-dimensional) Neumann problem centred at a boundary point $y$. The corresponding gradient field for interior points $y$ is
\begin{equation}\label{eq:alt-W} Y := \frac{1}{2}x -  \frac{1}{2}\frac{x-y}{|x-y|^k} -  \frac{1}{2|y|^{k-2}}\frac{x-\frac{y}{|y|^2}}{|x-\frac{y}{|y|^2}|^k} - \frac{k-2}{2}\int_0^{|y|} \frac{tx-\frac{y}{|y|}}{|tx-\frac{y}{|y|}|^k} \d t.\end{equation} Given $y\in B^n$ and a free boundary minimal submanifold $\Sigma^k$ in $(B^n, \pr B^n)$, applying essentially the same arguments as in \cite{Br12} yields \[\area(\Sigma^k) \geq \Theta^k(\Sigma, y) \area(B^k),\] recovering the same area estimate as Theorem \ref{thm:brendle}. In fact, for $y\in \Sigma \cap B^n$ we obtain the excess estimate
\[ \int_\Sigma \frac{|(x-y)^\perp|^2}{|x-y|^{k+2}} \leq \area(\Sigma^k) - \area(B^k).\] 
A similar excess estimate was used in \cite{BV18} to prove a gap for the area of non-flat free boundary minimal submanifolds in $(B^n,\pr B^n)$. 

Note that the approach using the alternative vector field (\ref{eq:alt-W}) also generalises to the varifold setting; indeed since $y$ is in the interior of $B^n$ one may proceed with the same cutoff method as for the classical monotonicity formula \cite{Si83}. 
\end{remark}

\subsection{Width of the ball}

We now deduce the width of the Euclidean ball from the above estimate, using the min-max characterisation Proposition \ref{thm:min-max}: 

\begin{proof}[Alternative proof of Theorem \ref{thm:main} ($K=0$)]
Up to scaling we may assume $R=1$. By considering a sweepout of $B^n$ by intersections with parallel $k$-dimensional affine subspaces, it is clear that $\omega^k(B^n) \leq \area(B^k)$. Let $V\in \mathcal{SV}_k(B^n)$ be a nontrivial, integral, free boundary stationary varifold realising the $k$-dimensional width $\omega^k(B^n) = \mathbf{M}(V)$. 
Theorem \ref{thm:fb-estimate} then implies $\omega^k(B^n) \geq \area(B^k)$ and hence the result. 
\end{proof}

\section{Estimates for free boundary minimal hypersurfaces} 
\label{sec:estimates}

In this section, we observe that the width gives a lower bound for the area of a free boundary minimal hypersurfaces in a space form ball. When $K=0$ this bound follows from the direct approaches of Brendle \cite{Br12}, and for $K>0$ one expects these to follow from the methods of Freidin-McGrath \cite{FM20, FM19}. However, to the author's knowledge, these sharp lower bounds have not appeared in the literature for $K<0$. 

\subsection{Stability of free boundary minimal hypersurfaces}

We briefly recall the concept of stability for free boundary minimal hypersurfaces. For more details the reader may consult \cite{ACS} and references therein. 

Let $M^n$ be a smooth, compact Riemannian manifold with boundary and suppose $\Sigma^{n-1}$ is a two-sided free boundary minimal hypersurface in $(M,\pr M)$. Let $\nu$ be the unit normal on $\Sigma$ and $\eta$ the outer unit conormal of $\pr \Sigma$ in $\Sigma$. Given any smooth function $\phi$ on $\Sigma$, the vector field $\phi \nu$ may be extended to a vector field on $M$ which is tangent to $\pr M$, and thereby generates a flow of $M$. The second variation of area along this flow is given by the quadratic form
\begin{equation}
\label{eq:stab-1}
Q(\phi,\phi) = \int_\Sigma \left( |\nabla^\Sigma \phi|^2 - (\Ric^M(\nu,\nu) + |A|^2)\phi^2\right) - \int_{\pr \Sigma} k^{\pr M}(\nu,\nu) \phi^2,
\end{equation}
where $A$ is the second fundamental form of $\Sigma$ in $M$, and $k^{\pr M}$ is the second fundamental form of $\pr M$ in $M$. If $\phi$ is smooth then integrating by parts gives 
\begin{equation}
\label{eq:stab-2}
Q(\phi,\phi) = -\int_\Sigma \phi L_\Sigma\phi + \int_{\pr \Sigma} \phi(\pr_\eta - k^{\pr M}(\nu,\nu)) \phi,
\end{equation} where the associated linear operator is $L_\Sigma= \Lap_\Sigma + \Ric^M(\nu,\nu) + |A|^2$, and $L^2(\Sigma)$ admits an orthonormal basis of eigenfunctions for the associated eigenvalue problem 

\[
\begin{cases}
L\phi = -\lambda \phi , & \text{on }  \Sigma\\
\pr_\eta \phi = k^{\pr M}(\nu,\nu)\phi , & \text{on } \pr \Sigma.
\end{cases}
\]

From this variational characterisation and the Harnack inequality, it follows that $\lambda_1$ is a simple eigenvalue and its associated eigenfunction $\phi_1$ may be taken to be strictly positive in the interior of $\Sigma$. 

The hypersurface $\Sigma$ is \textit{stable} if $Q(\phi,\phi)\geq 0$, or equivalently if the least eigenvalue $\lambda_1 = \inf_{\phi\neq 0} \frac{Q(\phi,\phi)}{\int_\Sigma \phi^2}$ of the above system is nonnegative. We say that $\Sigma$ is \textit{unstable} if it is not stable.

\subsection{Lower bounds via the width}

First, we observe that the width gives a lower bound in spaces which do not admit stable free boundary minimal hypersurfaces. The idea is, given a free boundary minimal hypersurface $\Sigma$, to construct a foliation whose maximal slice is $\Sigma$. 

\begin{proposition}
\label{prop:width-lb}
Let $M^n$ be a smooth, orientable compact Riemannian manifold with boundary. Suppose that there are no stable free boundary minimal hypersurfaces in $(M, \pr M)$. Then any embedded free boundary minimal hypersurface $\Sigma^{n-1}$ in $(M,\pr M)$ has $\area(\Sigma) \geq \omega^k(M)$. 
\end{proposition}
\begin{proof}
We may construct a sweepout of $M$ by deforming $\Sigma$ by its least eigenfunction and then running mean curvature flow. Indeed, since $\Sigma$ is unstable, the first stability eigenvalue $\lambda_1$ is strictly negative. Let $\phi_1>0$ be the associated eigenfunction on $\Sigma$. As in the closed setting, deformation by $\phi_1$ generates, for $t\in[-\epsilon,\epsilon]$ a family of hypersurfaces $(\Sigma_t,\pr\Sigma_t)$ in $(M,\pr M)$, with $\Sigma_0=\Sigma$. Moreover, since $\phi_1>0$ these hypersurfaces are pairwise disjoint and foliate a neighbourhood of $\Sigma$ in $M$. Since $\lambda_1<0$ and $L$ is the linearisation of the mean curvature at $\Sigma$, each $\Sigma_t$ is mean convex (with mean curvature vector pointing away from $\Sigma$). 

By mean convexity and since there are no stable free boundary minimal hypersurfaces, applying the mean curvature flow to $\Sigma_\epsilon$ and $\Sigma_{-\epsilon}$ will sweep out the remaining two components of $M \setminus \bigcup_{|t|\leq \epsilon} \Sigma_t$. (Strictly, one may use the level set formulation of free boundary mean curvature flow - see for instance \cite{EHIZ} or \cite{GS93}.) Since free boundary mean curvature flow decreases area, combining these with the $|t|\leq \epsilon$ foliation gives a sweepout of $M$ whose maximal area is $\area(\Sigma)$. By definition of width, this implies that $\area(\Sigma) \geq \omega^k(M)$. 
\end{proof}

The lower bound Proposition \ref{prop:estimate-intro} follows by checking that, indeed, space form balls do not admit stable free boundary minimal hypersurfaces:

\begin{corollary}
\label{cor:estimate}
Let $K\in \mathbb{R}$, $R>0$. If $K>0$, further assume that $R\leq \frac{\pi}{2\sqrt{K}}$. Then there are no stable free boundary minimal hypersurfaces in $(B^n_{R; K}, \pr B^n_{R; K})$. Consequently, any nontrivial embedded free boundary minimal hypersurface $\Sigma^{n-1}$ in $(B^n_{R; K}, \pr B^n_{R; K})$ has $\area(\Sigma) \geq \area( B^{n-1}_{R; K})$. 
\end{corollary}
\begin{proof}
By scaling we may assume $K\in\{-1,0,1\}$. For convenience we henceforth write $\Omega=  B^n_{R; K}$, and $|\Sigma|$ for the $k$-area of $\Sigma$ and $|\pr \Sigma|$ for the $(k-1)$ area of its boundary. 

Case 1: $K=1$. Then by the assumption on $R$ we have $k^{\pr \Omega} = \cot R \geq 0$. Since $\Ric^\Omega= n-1$, substituting into (\ref{eq:stab-1}) gives \[Q(1,1) = -(n-1) |\Sigma| - \int_\Sigma |A|^2 - \cot R \, |\pr\Sigma| <0.\] 

Case 2: $K=0$. Then $\Ric^\Omega=0$ and $k^{\pr \Omega} = \frac{1}{R}$, substituting into (\ref{eq:stab-1}) gives \[Q(1,1) = - \int_\Sigma |A|^2 - \frac{1}{R} |\pr\Sigma| <0.\] 

Case 3: $K=-1$. Then $\Ric^\Omega=-(n-1)$ and $k^{\pr \Omega} = \coth R$. Let $r$ be the distance function from the centre of the ball, and consider $\phi=\cosh r$. It is well-known that the distance function on hyperbolic space satisfies $\Hess \cosh r = (\cosh r)g,$ where $g$ is the ambient metric. If $\Sigma$ is minimal, it follows that $L_\Sigma \phi = |A|^2\phi$. Moreover by the free boundary condition, the outer conormal $\eta$ of $\pr \Sigma$ in $\Sigma$, is given by $\eta= \nabla r$. Then \[\phi(\pr_\eta - k^{\pr \Omega} (\nu,\nu))\phi = \cosh R \, \sinh R - \cosh^2 R\,\coth R= -\coth R.\]
Substituting into (\ref{eq:stab-2}) gives 

\[Q(\phi,\phi) = -\int_\Sigma |A|^2 \cosh^2 r -\coth R | \pr \Sigma| <0.\]

Thus in all cases, any free boundary minimal hypersurface $\Sigma$ in $(\Omega,\pr\Omega)$ is unstable. The area bound then follows from Proposition \ref{prop:width-lb} and Theorem \ref{thm:main}.
\end{proof}

We also note that Corollary \ref{cor:estimate} implies a sharp lower bound for the isoperimetric ratio (note that the $K=0$ case is a special case of that proven by Brendle \cite{Br12}):

\begin{corollary}
\label{cor:iso}
Let $K\leq 0$, $R>0$. Any embedded free boundary minimal hypersurface $\Sigma^{n-1}$ in $(B^n_{R; K}, \pr B^n_{R; K})$ satisfies 
\[ \frac{|\pr \Sigma|^{n-1}}{|\Sigma|^{n-2}} \geq \frac{|\pr B^{n-1}_{R;K}|^{n-1}}{|B^{n-1}_{R;K}|^{n-2}}. \]
\end{corollary}

\begin{proof}
The $K=0$ case is precisely \cite[Corollary 5]{Br12}. By scaling, we may henceforth assume $K=-1$. Let $r$ be the distance to the centre of the ball. As in \cite{FM19}, we may define a smooth radial vector field $\Phi = \phi(r) \frac{\pr}{\pr r}$, where $\phi(r) = \frac{1}{\sinh^{n-2} r} \int_0^r \sinh^{n-2} s\,\d s$. Moreover, given any minimal hypersurface $\Sigma^{n-1}$ this vector field satisfies $\div_\Sigma \Phi \geq 1$. 

If $\Sigma$ has free boundary, then by the divergence theorem it follows that \begin{equation}\label{eq:iso-1}|\Sigma| \leq \int_\Sigma \div_\Sigma \Phi = \int_{\pr\Sigma} \langle \Phi , \frac{\pr}{\pr r}\rangle = \phi(R) |\pr\Sigma|.\end{equation} On the other hand, recalling that $\beta_{n-2} = |\mathbb{S}^{n-2}|$, we have that $|\pr B^{n-1}_{R;K}| = \beta_{n-2} \sinh^{n-1} R$ and $|B^{n-1}_{R;K}| = \beta_{n-2} \int_0^R \sinh^{n-1} s\,\d s$. Using (\ref{eq:iso-1}) and Corollary \ref{cor:estimate} we conclude that 

\[ \frac{|\pr \Sigma|^{n-1}}{|\Sigma|^{n-2}} \geq \phi(R)^{n-1} |\Sigma| \geq \phi(R)^{n-1} |B^{n-1}_{R;K}| = \frac{|\pr B^{n-1}_{R;K}|^{n-1}}{|B^{n-1}_{R;K}|^{n-2}}  . \qedhere \]

\end{proof}

\bibliographystyle{amsalpha}
\bibliography{fb-width-v6}

\end{document}